\documentclass[12pt,reqno]{amsart}
\usepackage[top=2cm,bottom=2cm,right=2.5cm,left=2.5cm]{geometry}
\usepackage{amssymb}
\usepackage{amsmath, amsthm, amscd, amsfonts, amssymb, graphicx, color}
\usepackage[bookmarksnumbered, colorlinks, plainpages]{hyperref}

\usepackage{hyperref}

\newtheorem{theorem}{Theorem}[section]

\newtheorem{proposition}[theorem]{Proposition}

\theoremstyle{definition}
\newtheorem{definition}[theorem]{Definition}
\newtheorem{example}[theorem]{Example}

\theoremstyle{remark}
\newtheorem{remark}[theorem]{Remark}

\numberwithin{equation}{section}

\begin{document}
	
	\setcounter{page}{1}
	
	\title[Continuous $K$-biframes in Hilbert spaces]{Continuous $K$-biframes in Hilbert spaces}

	\author[A. Karara, M. Rossafi]{Abdelilah Karara$^{1}$ and Mohamed Rossafi$^2$$^{*}$}
	
	\address{$^{1}$Department of Mathematics Faculty of Sciences, University of Ibn Tofail, B.P. 133, Kenitra, Morocco}
	\email{\textcolor[rgb]{0.00,0.00,0.84}{ abdelilah.karara.sm@gmail.com}}
	\address{$^{2}$Department of Mathematics Faculty of Sciences, Dhar El Mahraz University Sidi Mohamed Ben Abdellah, Fes, Morocco}
	\email{\textcolor[rgb]{0.00,0.00,0.84}{mohamed.rossafi@usmba.ac.ma; mohamed.rossafi@uit.ac.ma}}
	
	
	
	\date{Received: xxxxxx; Revised: yyyyyy; Accepted: zzzzzz.
		\newline \indent $^{*}$ Corresponding author}
	
	\begin{abstract}
In this paper, we will introduce the concept of a continuous $K$-biframe for Hilbert spaces and we present various examples of continuous $K$-biframes. Furthermore, we investigate their characteristics from the perspective of operator theory by establishing various properties.
\end{abstract}
\maketitle

\baselineskip=12.4pt

\section{Introduction}
%=====================================
 
\smallskip\hspace{.6 cm}The notion  of frames in Hilbert spaces has been introduced by Duffin and Schaffer \cite{Duffin} in 1952 to research certain difficult nonharmonic Fourier series problems, following the fundamental paper \cite{Daubechies} by Daubechies, Grossman and Meyer, frame theory started to become popular, especially in the more specific context of Gabor frames and wavelet frames \cite{Gabor}. Currently, frames are frequently employed in cryptography, system modeling, quantum information processing, sampling,etc.\cite{Ferreira, Strohmer, Blocsli}.

The idea of pair frames, which refers to a pair of sequences in a Hilbert space, was first presented in \cite{Fer}. Parizi, Alijani and Dehghan \cite{MF} studied Biframe, which is a generalization of controlled frame in Hilbert space. The concept of a frame is defined from a single sequence but to define a biframe we will need two sequences. In fact, the concept of biframe is a generalization of controlled frames and a special case of pair frames.

 In this paper, we will introduce the concept of continuous $K$-biframes in Hilbert space and present some examples of this type of frame. Moreover, we investigate a characterization of continuous $K$-biframe by using the continuous biframe operator. Finally, in our exploration of continuous biframes, we investigate their characteristics from the perspective of operator theory by establishing various properties.

\section{Preliminaries}
%====================================
Throughout this paper, $\mathcal{H}$ represents a separable Hilbert space. The notation $\mathcal{B}(\mathcal{H},\mathcal{K})$ denotes the collection of all bounded linear operators from $\mathcal{H}$ to the Hilbert space $\mathcal{K}$. When $\mathcal{H} = \mathcal{K}$, this set is denoted simply as $\mathcal{B}(\mathcal{H})$. We will use $\mathcal{N}(\mathcal{T})$ and $\mathcal{R}(\mathcal{T})$ for the null  and range space of an operator $\mathcal{T}\in \mathcal{B}(\mathcal{H})$. Also $\mathrm{GL}^{+}(\mathcal{H})$ is the collection of all invertible, positive bounded linear operators acting on $\mathcal{H}$.
\begin{definition}\cite{Rahimi}
Let $\mathcal{H}$ be a complex Hilbert space and $(\Omega, \mu)$ be a measure space with positive measure $\mu$. A mapping $\mathcal{X} : \Omega \to \mathcal{H}$ is called a continuous frame with respect to $\left(\Omega, \mu\right)$ if
\begin{itemize}
\item[$(i)$] $\mathcal{X}$ is weakly-measurable, i.e., for all $f \in \mathcal{H}$, $\omega \mapsto \langle  f, \mathcal{X}(\omega)\rangle $ is a measurable function on $\Omega$,
\item[$(ii)$]there exist constants $0 < A \leq B < \infty$ such that
\end{itemize}
\[A \Vert f \Vert^{2} \leq \int_{\Omega}\vert\langle f, \mathcal{X}(\omega)\rangle\vert^{2} d\mu \leq B\Vert f \Vert^{2},\]
for all $f \in \mathcal{H}$. 
\end{definition}
The constants $A$ and $B$ are called continuous frame bounds. If $A = B$, then it is called a tight continuous frame. If the mapping $\mathcal{X}$ satisfies only the right inequality, then it is called continuous Bessel mapping with Bessel bound $B$.

Let $\mathcal{X} : \Omega \to \mathcal{H}$ be a continuous frame for $\mathcal{H}$. Then The synthesis operator $\mathcal{T}_{\mathcal{X}} : L^{2}\left(\Omega,\mu\right) \to \mathcal{H}$ weakly defined by
\[\langle  \mathcal{T}_{\mathcal{X}}(\varphi), h\rangle  = \int_{\Omega}\varphi(\omega)\langle  \mathcal{X}(\omega), f\rangle d\mu\]where $\varphi \in L^{2}\left(\Omega,\mu\right)$ and $f \in \mathcal{H}$ and its adjoint operator called the analysis operator $\mathcal{T}^{\ast}_{\mathcal{X}} : \mathcal{H} \to L^{2}\left(\Omega,\mu\right)$ is given by  
\[\mathcal{T}^{\ast}_{\mathcal{X}}\mathcal{X}(\omega) = \langle  f, \mathcal{X}(\omega)\rangle \;,\; f \in \mathcal{H},\;\; \omega \in \Omega.\]

The frame operator $S_{\mathcal{X}} : \mathcal{H} \to \mathcal{H}$ is weakly defined by
\[\langle  S_{\mathcal{X}}x, y\rangle  = \int_{\Omega}\langle  x, \mathcal{X}(\omega)\rangle \langle  \mathcal{X}(\omega), y\rangle d\mu, \;\forall x,y \in \mathcal{H}.\]

\begin{definition}\cite{Limaye} Let $\mathcal{H}$ be a Hilbert space, and suppose that $\mathcal{T} \in \mathcal{B}(\mathcal{H})$ has a closed range. Then there exists an operator $\mathcal{T}^{+} \in \mathcal{B}(\mathcal{H})$ for which
$$
N\left(\mathcal{T}^{+}\right)=\mathcal{R}(\mathcal{T})^{\perp}, \quad R\left(\mathcal{T}^{+}\right)=N(\mathcal{T})^{\perp}, \quad \mathcal{T} \mathcal{T}^{+} x=x, \quad f \in \mathcal{R}(\mathcal{T}) .
$$

We call the operator $\mathcal{T}^{+}$ the pseudo-inverse of $\mathcal{T}$. This operator is uniquely determined by these properties. In fact, if $\mathcal{T}$ is invertible, then we have $\mathcal{T}^{-1}=\mathcal{T}^{+}$.
\end{definition}
\begin{theorem}\cite{Douglas} \label{Doglas th} 
 Let $\mathcal{H}$ be a Hilbert space and $\mathcal{T}_{1},\mathcal{T}_{2}\in\mathcal{B}(\mathcal{H})$. The following statements are equivalent:
 \begin{enumerate}
 \item $\mathcal{R}(\mathcal{T}_{1})\subset\mathcal{R}(\mathcal{T}_{2})$
 \item  $\mathcal{T}_{1} \mathcal{T}_{1}^{\ast} \leq \lambda^2 \mathcal{T}_{2}\mathcal{T}_{2}^{\ast}$ for some $\lambda \geq 0$;
 \item $\mathcal{T}_{1} = \mathcal{T}_{2} U$ for some $U\in\mathcal{B}(\mathcal{H})$.
 
 \end{enumerate}
\end{theorem}
%=====================================
\section{Continuous $K$-biframe in Hilbert spaces }
%=====================================
In this section, we begin by presenting the definition of a continuous $K$-biframe in a Hilbert spaces, followed by a discussion of some of its properties. 
\begin{definition}\label{def1.01}
A pair $(\mathcal{X}, \mathcal{Y}) = \left(\mathcal{X} : \Omega \to \mathcal{H},\; \mathcal{Y} : \Omega \to \mathcal{H}\right)$ of mappings is called a continuous $K$-biframe for $\mathcal{H}$ with respect to $\left(\Omega, \mu\right)$ if:
\begin{itemize}
\item[$(i)$] $\mathcal{X}, \mathcal{Y}$ are weakly-measurable, i.e., for all $f \in \mathcal{H}$, $\omega \mapsto \langle   f, \mathcal{X}(\omega)\rangle $ and $\omega \mapsto \langle  f, \mathcal{Y}(\omega)\rangle $ are measurable functions on $\Omega$,
\item[$(ii)$]there exist constants $0 < A \leq B < \infty$ such that for all $f \in \mathcal{H}$:
\end{itemize}
\begin{align}
A\Vert K^{\ast} f \Vert^{2} \leq \int_{\Omega}\langle  f, \mathcal{X}(\omega)\rangle \langle  \mathcal{Y}(\omega), f\rangle d\mu \leq B\Vert f \Vert^{2}.\label{3.eqq3.11}
\end{align}
The constants $A$ and $B$ are called continuous $K$-biframe bounds. If $A = B$, then it is called a tight continuous $K$-biframe and if $A = B = 1$, then it is called Parseval continuous $K$-biframe . 

If $(\mathcal{X}, \mathcal{Y})$ satisfies only the right inequality (\ref{3.eqq3.11}), then it is called continuous $K$-biframe Bessel mapping with Bessel bound $B$. 
\end{definition}
\begin{remark}
Let $\mathcal{X} : \Omega \to \mathcal{H}$ be a mapping. Consequently, in light of the Definition \ref{def1.01}, we express that
\begin{itemize}
\item[$(i)$]If $(\mathcal{X}, \mathcal{X})$ is a continuous $K$-biframe for $\mathcal{H}$, then $\mathcal{X}$ is a continuous frame for $\mathcal{H}$.
\item[$(ii)$]If $P \in GL^{+}(\mathcal{H})$, $(\mathcal{X}, P\mathcal{X})$ is a continuous $K$-biframe for $\mathcal{H}$, then $\mathcal{X}$ is a $P$--controlled continuous frame for $\mathcal{H}$,
\item[$(iii)$]If $P, Q \in GL^{+}(\mathcal{H})$, $(P\mathcal{X}, Q\mathcal{X})$ is a continuous $K$-biframe for $\mathcal{H}$, then $\mathcal{X}$ is a $(P, Q)$--controlled continuous frame for $\mathcal{H}$.  
\end{itemize} 
\end{remark}

We now provide some examples that verify the description given above.
\begin{example}\label{3exm3.11}
Let $\left\{e_k\right\}_{i=1}^{\infty}$ be an orthonormal basis for $\mathcal{H}$. We consider two sequences $\mathcal{X}=\left\{f_i\right\}_{i=1}^{\infty}$ and $\mathcal{Y}=\left\{g_i\right\}_{i=1}^{\infty}$ defined as follows: 
\begin{align*}
&\mathcal{X} = \left\{e_{1}, e_{1}, e_{1}, e_{2}, e_{3}, \cdots\cdots\right\},\\
&\mathcal{Y} = \left\{0, e_{1}, e_{1},  e_{2}, e_{3}, \cdots\cdots\right\}.
\end{align*}
We have $\Omega = \bigcup_{i = 1}^{\infty}\Omega_{i}$, where $\left\{\Omega_{i}\right\}_{i = 1}^{\infty}$ is a sequence of disjoint measurable subsets of $\Omega$ with $\mu\left(\Omega_{i}\right) < \infty$. For any $\omega \in \Omega$, we define the mappings $\mathcal{X} : \Omega \to \mathcal{H}$ by $\mathcal{X}(\omega) = \left(\mu\left(\Omega_{i}\right)\right)^{-1}f_{i}$ and $\mathcal{Y} : \Omega \to \mathcal{H}$ by $\mathcal{Y}(\omega) =\left(\mu\left(\Omega_{i}\right)\right)^{-1}g_{i}$ and $K: \mathcal{H} \to \mathcal{H}$ by $$Ke_{1}=e_{1},\;Ke_{2}=e_{1},\;Ke_{3}=e_{1},\;Ke_{4}=e_{2},\;Ke_{5}=e_{3},\cdots.$$ Then for $f \in \mathcal{H}$, 
\begin{align*}
\int_{\Omega}\langle  f, \mathcal{X}(\omega)\rangle \langle  \mathcal{X}(\omega),f \rangle d\mu &=\sum_{i = 1}^{\infty}\int_{\Omega_{i}}\langle f, f_{i}\rangle \langle f_{i},f \rangle d\mu\\
& =2 \langle f, e_{1}\rangle \langle f, e_{1}\rangle + \sum_{i = 1}^{\infty} \langle f, e_{i} \rangle \langle e_{i}, f \rangle\\
&=2\langle f, e_{1}\rangle \langle f, e_{1}\rangle + \Vert f \Vert^{2}. \\
& \leq 3 \Vert f \Vert^{2}
\end{align*} 
So $$ \Vert K^{\ast} f \Vert^{2} \leq \int_{\Omega}\langle  f, \mathcal{X}(\omega)\rangle \langle  \mathcal{X}(\omega),f \rangle d\mu \leq 3 \Vert f \Vert^{2} $$
Therefore, $\mathcal{X}$ is a continuous $K$-frame for $\mathcal{H}$ with bounds $1$ and $3$. Similarly, we have 
\begin{align*}
\int_{\Omega}\langle  f, \mathcal{Y}(\omega)\rangle \langle  \mathcal{Y}(\omega),f \rangle d\mu &=\sum_{i = 1}^{\infty}\int_{\Omega_{i}}\langle f, g_{i}\rangle \langle g_{i},f \rangle d\mu\\
& =\langle f, e_{1}\rangle \langle f, e_{1}\rangle + \sum_{i = 1}^{\infty} \langle f, e_{i} \rangle \langle e_{i}, f \rangle\\
&=\langle f, e_{1}\rangle \langle f, e_{1}\rangle + \Vert f \Vert^{2}. \\
& \leq 2 \Vert f \Vert^{2}
\end{align*} 
So $$ \Vert K^{\ast} f \Vert^{2} \leq \int_{\Omega}\langle  f, \mathcal{Y}(\omega)\rangle \langle  \mathcal{Y}(\omega),f \rangle d\mu \leq 2 \Vert f \Vert^{2} $$
Therefore $\mathcal{Y}$ is a continuous $K$-frame for $\mathcal{H}$ with bounds $1$ and $2$.

Now, for $f \in \mathcal{H}$, we have
\begin{align*}
&\int_{\Omega}\langle  f, \mathcal{X}(\omega)\rangle \langle  \mathcal{Y}(\omega), f\rangle d\mu =\sum_{i = 1}^{\infty}\int_{\Omega_{i}}\left <f, f_{i}\right >\left <g_{i}, f\right >d\mu\\
&=\left <f, e_{1}\right >\left <e_{1}, f\right > + \left <f, e_{1}\right >\left <e_{1}, f\right > + \left <f, e_{2}\right >\left <e_{2}, f\right >+\cdots\\
&=\left <f, e_{1}\right >\left <e_{1}, f\right > + \sum_{i = 1}^{\infty} \langle f, e_{i} \rangle \langle e_{i}, f \rangle\\
&=\left <f, e_{1}\right >\left <e_{1}, f\right > + \Vert f \Vert^{2}\\
&\leq 2 \Vert f \Vert^{2}.
\end{align*} 
So $$ \Vert K^{\ast} f \Vert^{2} \leq \int_{\Omega}\langle  f, \mathcal{X}(\omega)\rangle \langle  \mathcal{Y}(\omega),f \rangle d\mu \leq 2 \Vert f \Vert^{2} $$
Thus, $(\mathcal{X}, \mathcal{Y})$ is a continuous $K$-biframe for $\mathcal{H}$ with bounds $1$ and $2$.
\end{example}
\begin{example}  Assume that $\mathcal{S}=\left\{\left(\begin{array}{ll}a & 0 \\ 0 & b\end{array}\right): a, b \in \mathbb{R}\right\}$,with the inner product:
$$
\begin{aligned}
\langle., .\rangle: \mathcal{S} \times \mathcal{S} & \rightarrow \mathbb{R} \\
(M, N) & \longmapsto MN^t .
\end{aligned}
$$
for all $M, N \in\mathcal{S}$. It is then straightforward to verify that $\langle\cdot, \cdot\rangle$ forms a real inner product on $\mathcal{S}$. Now, let's consider a measure space $(\Omega=[0,1], \mu)$ where $\mu$ is the Lebesgue measure. Define $\mathcal{X} : \Omega \to \mathcal{S}$ by
\[
\mathcal{X}(\omega) = 
\begin{pmatrix}
\;2 \omega & 0\\
\;0           &   1-\omega\\
\end{pmatrix}
,\; \omega \in \Omega
\] 
and $\mathcal{Y} : \Omega \to \mathcal{S}$ by
\[
\mathcal{Y}(\omega) = 
\begin{pmatrix}
\;3\omega & 0\\
\;0           & \omega+1\\
\end{pmatrix}
,\; \omega \in \Omega
\]
For all $M \in \mathcal{S}$, it's straightforward to verify that the functions $\omega \mapsto\langle M, \mathcal{X}(\omega)\rangle$ and $\omega \mapsto\langle M, \mathcal{Y}(\omega)\rangle$ are measurable on $\Omega$. Let's define $$K: \mathcal{S} \rightarrow \mathcal{S} \quad\text{by}  \quad K M=\sqrt{2} M \quad\text{for all}  \quad M \in \mathcal{S}.$$ Then, it's easy to verify that $\Vert K^* M\Vert^2=2\Vert M\Vert^2$.

For every $M_f=\left(\begin{array}{ll}a & 0 \\ 0 & b\end{array}\right) \in \mathcal{S}$, we have
$$
\begin{aligned}
\int_{\Omega}\langle M_f, \mathcal{X}(\omega)\rangle\langle \mathcal{Y}(\omega), M_f\rangle d \mu(\omega) & =\int_{[0,1]}\left\langle\left(\begin{array}{cc}
a & 0 \\
0 & b
\end{array}\right),\left(\begin{array}{cc}
2 \omega & 0 \\
0 & 1-\omega
\end{array}\right)\right\rangle\\ & \quad\quad\quad\quad\quad\left\langle\left(\begin{array}{cc}
3 \omega & 0 \\
0 & \omega+1
\end{array}\right),\left(\begin{array}{ll}
a & 0 \\
0 & b
\end{array}\right)\right\rangle d \mu(\omega) \\
& =\int_{[0,1]}\left(\begin{array}{cc}
2 \omega a & 0 \\
0 & (1-\omega) b
\end{array}\right)\left(\begin{array}{cc}
3 \omega a & 0 \\
0 & (\omega+1) b
\end{array}\right) d \mu(\omega) \\
& =\int_{[0,1]}\left(\begin{array}{cc}
6 \omega^2 & 0 \\
0 & 1-\omega^2
\end{array}\right)\left(\begin{array}{cc}
a^2 & 0 \\
0 & b^2
\end{array}\right) d \mu(\omega) \\
& =\left(\begin{array}{cc}
a^2 & 0 \\
0 & b^2
\end{array}\right) \int_{[0,1]}\left(\begin{array}{cc}
6 \omega^2 & 0 \\
0 & 1-\omega^2
\end{array}\right) d \mu(\omega) \\
& =\left(\begin{array}{cc}
2 & 0 \\
0 & \frac{2}{3}
\end{array}\right)\left(\begin{array}{cc}
a^2 & 0 \\
0 & b^2
\end{array}\right),
\end{aligned}
$$
consequently
 $$ \dfrac{1}{3}\Vert K^{\ast}f \Vert^{2} \leq \int_{\Omega}\langle  f, \mathcal{X}(\omega)\rangle \langle  \mathcal{Y}(\omega),f \rangle d\mu \leq 2 \Vert f \Vert^{2} $$
Therefore, $(\mathcal{X}, \mathcal{Y})$ is a continuous $K$-biframe for $\mathcal{H}$ with bound $\dfrac{1}{3}$ and $2$. 
\end{example}
Next, we introduce the continuous biframe operator and provide some of its properties.
\begin{definition}
Let $(\mathcal{X}, \mathcal{Y}) = \left(\mathcal{X} : \Omega \to \mathcal{H},\; \mathcal{Y} : \Omega \to \mathcal{H}\right)$ be a continuous biframe for $\mathcal{H}$ with respect to $\left(\Omega, \mu\right)$. Then the continuous biframe operator $S_{\mathcal{X}, \mathcal{Y}} : \mathcal{H} \to \mathcal{H}$ is defined by
\[S_{\mathcal{X}, \mathcal{Y}}f = \int_{\Omega}\langle  f, \mathcal{X}(\omega)\rangle \mathcal{Y}(\omega)d\mu,\]
for all $f \in \mathcal{H}$.
\end{definition}

For every $f \in \mathcal{H}$, we have 
\begin{align}
\langle  S_{\mathcal{X}, \mathcal{Y}}f, f\rangle  = \int_{\Omega}\langle  f, \mathcal{X}(\omega)\rangle \langle  \mathcal{Y}(\omega), f\rangle d\mu.\label{3.eqq3.12}
\end{align}
This implies that, for each $f \in \mathcal{H}$,
\[A\Vert f \Vert^{2} \leq \langle  S_{\mathcal{X}, \mathcal{Y}}f, f\rangle  \leq  B\Vert f \Vert^{2}.\]
Hence $AI \leq S_{\mathcal{X}, \mathcal{Y}} \leq BI$, where $I$ is the identity operator on $\mathcal{H}$. consequently, $S_{\mathcal{X}, \mathcal{Y}}$ is positive and invertible.

\begin{proposition}
Let $S_{\mathcal{X}, \mathcal{Y}}$ and $S_{\mathcal{Y}, \mathcal{X}}$ be continuous biframe operators such that $S_{\mathcal{X}, \mathcal{Y}}=S_{\mathcal{Y}, \mathcal{X}}$. Then the pair $(\mathcal{X}, \mathcal{Y})$ is a continuous $K$-biframe for $\mathcal{H}$ with respect to $\left(\Omega, \mu\right)$ if and only if  $(\mathcal{Y}, \mathcal{X})$ is a continuous $K$-biframe for $\mathcal{H}$ with respect to $\left(\Omega, \mu\right)$
\end{proposition}
\begin{proof}
Let $(\mathcal{X}, \mathcal{Y})$ is a continuous $K$-biframe for $\mathcal{H}$ with bounds $A$ and $B$. Then for every $f \in \mathcal{H}$, we have 
\[A\Vert K^{\ast} f \Vert^{2} \leq\langle S_{\mathcal{X}, \mathcal{Y}}f,f\rangle= \int_{\Omega}\langle  f, \mathcal{X}(\omega)\rangle \langle  \mathcal{Y}(\omega), f\rangle d\mu \leq B\Vert f \Vert^{2}.\] 
Since $S_{\mathcal{X}, \mathcal{Y}}=S_{\mathcal{Y}, \mathcal{X}}$ we have 
$$\langle S_{\mathcal{Y}, \mathcal{X}}f,f\rangle=\langle S_{\mathcal{X}, \mathcal{Y}}f,f\rangle=\int_{\Omega}\langle  f, \mathcal{Y}(\omega)\rangle \langle  \mathcal{X}(\omega), f\rangle d\mu$$
Thus, for each $f \in \mathcal{H}$, we have 
\[A\Vert K^{\ast}f \Vert^{2} \leq \int_{\Omega}\langle  f, \mathcal{Y}(\omega)\rangle \langle  \mathcal{X}(\omega), f\rangle d\mu \leq B\Vert f \Vert^{2}.\]
Therefore, $(\mathcal{Y}, \mathcal{X})$ is a continuous $K$-biframe for $\mathcal{H}$.

Likewise, we can establish the converse part of this theorem.
\end{proof}
In the following theorem, we establish a characterization of a continuous $K$-biframe by utilizing its continuous biframe operator.

\begin{theorem}
Let $(\mathcal{X}, \mathcal{Y})$ be a continuous biframe for $\mathcal{H}$ with respect to $\left(\Omega, \mu\right)$.Then $(\mathcal{X}, \mathcal{Y})$ is a continuous $K$-biframe  with bounds $A$ and $B$ for $\mathcal{H}$ if and only if $S_{\mathcal{X}, \mathcal{Y}} \geq AKK^{\ast}$, where $S_{\mathcal{X}, \mathcal{Y}}$ is the continuous biframe operator for $(\mathcal{X}, \mathcal{Y})$.
\end{theorem}

\begin{proof}
Let $(\mathcal{X}, \mathcal{Y})$ is a continuous $K$-biframe for $\mathcal{H}$ with bounds $A$ and $B$. Then using (\ref{3.eqq3.11}) and (\ref{3.eqq3.12}), for each $f \in \mathcal{H}$, we get 
\[A\Vert K^{\ast}f \Vert^{2} \leq \langle  S_{\mathcal{X}, \mathcal{Y}}f, f\rangle  = \int_{\Omega}\langle  f, \mathcal{X}(\omega)\rangle \langle  \mathcal{Y}(\omega), f\rangle d\mu \leq  B\Vert f \Vert^{2}.\]
Thus
\[A\langle  K^{\ast}f, K^{\ast}f\rangle  \leq \langle  S_{\mathcal{X}, \mathcal{Y}}f, f\rangle  ,\]
Hence $$S_{\mathcal{X}, \mathcal{Y}} \geq A KK^{\ast}.$$

Conversely, assume that $S_{\mathcal{X}, \mathcal{Y}} \geq A KK^{\ast}$. Then, for every $f \in \mathcal{H}$, we have
\[A\Vert K^{\ast}f \Vert^{2} \leq  \langle  S_{\mathcal{X}, \mathcal{Y}}f, f\rangle  = \int_{\Omega}\langle  f, \mathcal{X}(\omega)\rangle \langle  \mathcal{Y}(\omega), f\rangle d\mu.\] 
Since $(\mathcal{X}, \mathcal{Y})$ is a continuous biframe for $\mathcal{H}$. Therefore, $(\mathcal{X}, \mathcal{Y})$ is a continuous $K$-biframe for $\mathcal{H}$.
\end{proof}

Furthermore, we provide a characterization of a continuous biframe with the assistance of an invertible operator on $\mathcal{H}$.
\begin{theorem}\label{3.thm3.39}
Let $\mathcal{T}\in \mathcal{B}(\mathcal{H})$ be invertible on $\mathcal{H}$. Then the
following statements are equivalent:
\begin{enumerate}
\item $(\mathcal{X}, \mathcal{Y})$ is a continuous $K$-biframe for $\mathcal{H}$ with respect to $\left(\Omega, \mu\right)$
\item $(\mathcal{T}\mathcal{X}, \mathcal{T}\mathcal{Y})$ is a continuous $K$-biframe for $\mathcal{H}$ with respect to $\left(\Omega, \mu\right)$.
\end{enumerate}
\end{theorem}

\begin{proof}
(1)$\Rightarrow$(2) For each $f \in \mathcal{H}$, $$ \omega \mapsto\langle  f, \mathcal{T}\mathcal{X}(\omega)\rangle   $$ and $$\omega \mapsto \langle  f, \mathcal{T}\mathcal{Y}(\omega)\rangle  $$ are measurable functions on $\Omega$. Let $(\mathcal{X}, \mathcal{Y})$ is a continuous biframe for $\mathcal{H}$ with bounds $A$ and $B$ and $\mathcal{T}\in \mathcal{B}(\mathcal{H})$. for $f \in \mathcal{H}$, we have  
\begin{align*}
\int_{\Omega}\langle  f, \mathcal{T}\mathcal{X}(\omega)\rangle \langle  \mathcal{T}\mathcal{Y}(\omega), f\rangle d\mu &= \int_{\Omega}\langle  \mathcal{T}^{\ast}f, \mathcal{X}(\omega)\rangle \langle  \mathcal{Y}(\omega), \mathcal{T}^{\ast}f\rangle d\mu\\
&\leq B\Vert  \mathcal{T}^{\ast}f\Vert^{2} \\
&\leq B \Vert \mathcal{T}\Vert^{2} \Vert f\Vert^{2}.
\end{align*}

On the other hand, Since $\mathcal{T}\in \mathcal{B}(\mathcal{H})$ is invertible, for each $f \in \mathcal{H}$, we have
\begin{align*}
\Vert K^{\ast} f \Vert^{2} &=\Vert \left (\mathcal{T}\mathcal{T}^{- 1}\right )^{\ast} K^{\ast}f\Vert^{2} \\
&=\Vert \left (\mathcal{T}^{- 1}\right )^{\ast}\mathcal{T}^{\ast} K^{\ast}f\Vert^{2} \\
&\leq \left\|\left (\mathcal{T}^{- 1}\right )^{\ast}\right\|^{2} \Vert\mathcal{T}^{\ast} K^{\ast}f\Vert^{2}.
\end{align*}
Consequently, for each $f \in \mathcal{H}$, we have 
\begin{align*}
\int_{\Omega}\langle  f, \mathcal{T}\mathcal{X}(\omega)\rangle \langle  \mathcal{T}\mathcal{Y}(\omega), f\rangle d\mu &= \int_{\Omega}\langle  \mathcal{T}^{\ast}f, \mathcal{X}(\omega)\rangle \langle  \mathcal{Y}(\omega), \mathcal{T}^{\ast}f\rangle d\mu\\
&\geq A\Vert\mathcal{T}^{\ast} K^{\ast}f\Vert^{2}\\
&\geq A\left\|\left (\mathcal{T}^{- 1}\right )^{\ast}\right\|^{- 2}\Vert  K^{\ast}f \Vert^{2}.
\end{align*}
Hence, $(\mathcal{T}\mathcal{X}, \mathcal{T}\mathcal{Y})$ is a continuous $K$-biframe for $\mathcal{H}$ with bounds $A\left\|\left (\mathcal{T}^{- 1}\right )^{\ast}\right\|^{- 2}$ and $B\|\mathcal{T}\|^{2}$.

(2)$\Rightarrow$(1), Assume that $(\mathcal{T}\mathcal{X}, \mathcal{T}\mathcal{Y})$ is a continuous $K$-biframe for $\mathcal{H}$ with bounds $A$ and $B$. Now, for each $f \in \mathcal{H}$, we have
\begin{align*}
A\|\mathcal{T}^{\ast}\|^{-2}\Vert  K^{\ast}f \Vert^{2} &= A\|\mathcal{T}^{\ast}\|^{-2}\Vert \left (\mathcal{T}\mathcal{T}^{- 1}\right )^{\ast} K^{\ast}f\Vert^{2}\\
 &\leq A\Vert\left(\mathcal{T}^{- 1}\right)^{\ast}K^{\ast}f\Vert^{2} \\
&\leq \int_{\Omega}\langle  \left(\mathcal{T}^{- 1}\right)^{\ast}f, \mathcal{T}\mathcal{X}(\omega)\rangle \langle  \mathcal{T}\mathcal{Y}(\omega), \left(\mathcal{T}^{- 1}\right)^{\ast}f\rangle d\mu \\
&= \int_{\Omega}\langle  \mathcal{T}^{\ast}\left(\mathcal{T}^{- 1}\right)^{\ast}f, \mathcal{X}(\omega)\rangle \langle  \mathcal{Y}(\omega), \mathcal{T}^{\ast}\left(\mathcal{T}^{- 1}\right)^{\ast}f\rangle d\mu\\
&= \int_{\Omega}\langle  f, \mathcal{X}(\omega)\rangle \langle  \mathcal{Y}(\omega), f\rangle d\mu.   
\end{align*}
On the other hand, for each $f \in \mathcal{H}$, we have
\begin{align*}
&\int_{\Omega}\langle  f, \mathcal{X}(\omega)\rangle \langle  \mathcal{Y}(\omega), f\rangle d\mu\\
&=\int_{\Omega}\langle  \mathcal{T}^{\ast}\left(\mathcal{T}^{- 1}\right)^{\ast}f, \mathcal{X}(\omega)\rangle \langle  \mathcal{Y}(\omega), \mathcal{T}^{\ast}\left(\mathcal{T}^{- 1}\right)^{\ast}f\rangle d\mu\\
&=\int_{\Omega}\langle  \left(\mathcal{T}^{- 1}\right)^{\ast}f, \mathcal{T}\mathcal{X}(\omega)\rangle \langle  \mathcal{T}\mathcal{Y}(\omega), \left(\mathcal{T}^{- 1}\right)^{\ast}f\rangle d\mu\\
&\leq B\Vert\left(\mathcal{T}^{- 1}\right)^{\ast}f\Vert^{2}\\ 
&\leq B\left\|\left(\mathcal{T}^{- 1}\right)^{\ast}\right\|^{2}\Vert f \Vert^{2}.
\end{align*}
Therefore, $(\mathcal{X}, \mathcal{Y})$ is a continuous $K$-biframe for $\mathcal{H}$ with bounds $A\|\mathcal{T}^{\ast}\|^{-2}$ and $B\left\|\left(\mathcal{T}^{- 1}\right)^{\ast}\right\|^{2}$.
\end{proof}
In the following proposition we will require a necessary condition for the operator $\mathcal{T}$ for which $ (\mathcal{X}, \mathcal{Y})$  will be $\mathcal{T}$-biframe for $\mathcal{H}$.
 \begin{proposition}
   Let $(\mathcal{X}, \mathcal{Y})$ be a continuous $K$-biframe for $\mathcal{H}$. Assume that $\mathcal{T} \in \mathcal{B}(\mathcal{H})$ with $\mathcal{R}(\mathcal{T}) \subseteq$ $\mathcal{R}(K)$. Then $(\mathcal{X}, \mathcal{Y})$ is a continuous $\mathcal{T}$-biframe for $\mathcal{H}$.
\end{proposition}
\begin{proof}
Suppose that $(\mathcal{X}, \mathcal{Y})$ is a continuous $K$-biframe for $\mathcal{H}$. Then there are positive constants $0<A \leq B<\infty$ such that
$$
A\left\|K^* f\right\|^2 \leq  \int_{\Omega}\langle  f, \mathcal{X}(\omega)\rangle \langle  \mathcal{Y}(\omega), f\rangle d\mu \leq B\|x\|^2, \text { for all } f \in \mathcal{H} .
$$

Since $R(\mathcal{T}) \subseteq \mathcal{R}(K)$, by Theorem \ref{Doglas th}, there exists $\alpha>0$ such that $\mathcal{T} \mathcal{T}^{\ast} \leq \alpha^2 K K^*$.

Hence, 
$$
\frac{A}{\alpha^2}\left\|\mathcal{T}^{\ast} f\right\|^2 \leq A\left\|K^* f\right\|^2 \leq  \int_{\Omega}\langle  f, \mathcal{X}(\omega)\rangle \langle  \mathcal{Y}(\omega), f\rangle d\mu \leq B\|x\|^2, \text { for all } f \in \mathcal{H} .
$$

Hence $(\mathcal{X}, \mathcal{Y})$ is a continuous $\mathcal{T}$-biframe for $\mathcal{H}$.
\end{proof}

\begin{theorem}
 Suppose $K \in \mathcal{B}(\mathcal{H})$ has a closed range. The continuous biframe operator of a continuous $K$-biframe is invertible on the subspace $\mathcal{R}(K)$ of $\mathcal{H}$.
\end{theorem}
\begin{proof}
 Assume that $(\mathcal{X}, \mathcal{Y})$ be a continuous $K$-biframe for $\mathcal{H}$.  Then there are positive constants $0<A \leq B<\infty$ such that
$$
A\left\|K^* f\right\|^2 \leq  \int_{\Omega}\langle  f, \mathcal{X}(\omega)\rangle \langle  \mathcal{Y}(\omega), f\rangle d\mu \leq B\|x\|^2, \text { for all } f \in \mathcal{H} .
$$

Since $\mathcal{R}(K)$ is closed, then $K K^{+} x=x$, for all $f \in \mathcal{R}(K)$. That is,
$$
\left.K K^{+}\right|_{\mathcal{R}(K)}=I_{\mathcal{R}(K)},
$$
we have $I_{\mathcal{R}(K)}^{\ast}=\left(\left.K^{+}\right|_{\mathcal{R}(K)}\right)^{\ast} K^*$.
For any $f \in \mathcal{R}(K)$, we obtain
$$
\|x\|^{2}=\left\|\left(\left.K K^{+}\right|_{\mathcal{R}(K)}\right)^{\ast} x \right\|^{2}=\left\|\left(\left.K^{+}\right|_{\mathcal{R}(K)}\right)^{\ast} K^* f\right\| \leq\left\|K^{+}\right\|^{2} \cdot\left\|K^* f\right\|^{2}.
$$
Thus $$\left\|K^* f\right\|^2 \geq\left\|K^{+}\right\|^{-2}\|x\|^2.$$ So, we have
$$
\int_{\Omega}\langle  f, \mathcal{X}(\omega)\rangle \langle  \mathcal{Y}(\omega), f\rangle d\mu \geq A\left\|K^* f\right\|^2 \geq A\left\|K^{+}\right\|^{-2}\|x\|^2, \text { for all } f \in \mathcal{R}(K) .
$$

Therefore, based on the definition of a continuous $K$-biframe, we have
$$
A\left\|K^{+}\right\|^{-2}\|x\|^2 \leq \int_{\Omega}\langle  f, \mathcal{X}(\omega)\rangle \langle  \mathcal{Y}(\omega), f\rangle d\mu \leq B\|x\|^2, \text { for all } f \in \mathcal{R}(K) .
$$

Hence
$$
A\left\|K^{+}\right\|^{-2}\|x\| \leq\|\left.S_{\mathcal{X},\mathcal{Y}}\right|_{\mathcal{R}(K)} x\| \leq B\|x\|, \text { for all } f \in \mathcal{R}(K) .
$$

Consequently, $\left.S_{\mathcal{X},\mathcal{Y}}\right|_{\mathcal{R}(K)}: \mathcal{R}(K) \rightarrow R(S)$ is a bounded linear operator and invertible on $\mathcal{R}(K)$.
\end{proof}
\begin{theorem}
 Let $K \in \mathcal{B}(\mathcal{H})$ be with a dense range. Suppose that $(\mathcal{X}, \mathcal{Y}) = \left(\mathcal{X} : \Omega \to \mathcal{H},\; \mathcal{Y} : \Omega \to \mathcal{H}\right)$ be a continuous $K$-biframe and $\mathcal{T} \in \mathcal{B}(\mathcal{H})$ have a closed range.\\ If $(\mathcal{T} \mathcal{X}, \mathcal{T} \mathcal{Y}) = \left(\mathcal{T} \mathcal{X} : \Omega \to \mathcal{H},\; \mathcal{T} \mathcal{Y} : \Omega \to \mathcal{H}\right)$ is a continuous $K$-biframe for $\mathcal{H}$, then $\mathcal{T}$ is surjective.
\end{theorem}
\begin{proof}
Suppose That  $(\mathcal{T} \mathcal{X}, \mathcal{T} \mathcal{Y}) = \left(\mathcal{T} \mathcal{X} : \Omega \to \mathcal{H},\; \mathcal{T} \mathcal{Y} : \Omega \to \mathcal{H}\right)$ is a continuous $K$-biframe for $\mathcal{H}$ with frame bounds $A$ and $B$. Then for each $f \in \mathcal{H}$,
$$
A\left\|K^* f\right\|^2 \leq  \int_{\Omega}\langle  f, \mathcal{T}\mathcal{X}(\omega)\rangle \langle \mathcal{T} \mathcal{Y}(\omega), f\rangle d\mu   \leq B\|x\|^2 .
$$
Hence 
\begin{equation}\label{EQ41}
A\left\|K^* f\right\|^2 \leq   \int_{\Omega}\langle  \mathcal{T}^{\ast}f, \mathcal{X}(\omega)\rangle \langle  \mathcal{Y}(\omega), \mathcal{T}^{\ast}f\rangle d\mu  \leq B\|x\|^2 .
\end{equation}
Since $K$ has a dense range, $\mathcal{H}=\overline{\mathcal{R}(K)}$, so $K^*$ is injective. From (\ref{EQ41}), $\mathcal{T}^{\ast}$ is injective since $N\left(\mathcal{T}^{\ast}\right) \subseteq \textbf{N}\left(K^*\right)$. Moreover, $$\mathcal{R}(\mathcal{T})=\mathcal{N}\left(\mathcal{T}^{\ast}\right)^{\perp}=\mathcal{H}.$$ Therefore $\mathcal{T}$ is surjective.
\end{proof}
\begin{theorem}
 Let $K \in \mathcal{B}(\mathcal{H})$ and let  $(\mathcal{X}, \mathcal{Y})$ be a continuous $K$-biframe for $\mathcal{H}$. If $\mathcal{T} \in \mathcal{B}(\mathcal{H})$ has a closed range with $\mathcal{T} K=K \mathcal{T}$,  then $(\mathcal{T} \mathcal{X}, \mathcal{T} \mathcal{Y}) = \left(\mathcal{T} \mathcal{X} : \Omega \to \mathcal{H},\; \mathcal{T} \mathcal{Y} : \Omega \to \mathcal{H}\right)$ is a continuous $K$-biframe for $\mathcal{R}(\mathcal{T})$.
\end{theorem}
\begin{proof}
 Since $\mathcal{T}\in \mathcal{B}(\mathcal{H})$ has a closed range. Then for each $f \in \mathcal{R}(\mathcal{T})$, $$K^* f=\left(\mathcal{T}^{+}\right)^{\ast} \mathcal{T}^{\ast} K^* f,$$ so we have
$$
\left\|K^* f\right\|=\left\|\left(\mathcal{T}^{+}\right)^{\ast} \mathcal{T}^{\ast} K^* f\right\| \leq\left\|\left(\mathcal{T}^{+}\right)^{\ast}\right\|\left\|\mathcal{T}^{\ast} K^* f\right\| .
$$

Hence $$\left\|\left(\mathcal{T}^{+}\right)^{\ast}\right\|^{-1}\left\|K^* f\right\| \leq\left\|\mathcal{T}^{\ast} K^* f\right\|.$$
 Since $(\mathcal{X}, \mathcal{Y})$ is a continuous $K$-biframe with frame bounds $A, B$, then for each $f \in \mathcal{R}(\mathcal{T})$, we have
$$
\begin{aligned}
 \int_{\Omega}\langle  f, \mathcal{T}\mathcal{X}(\omega)\rangle \langle \mathcal{T} \mathcal{Y}(\omega), f\rangle d\mu  & =\int_{\Omega}\langle  \mathcal{T}^{\ast}f, \mathcal{X}(\omega)\rangle \langle  \mathcal{Y}(\omega), \mathcal{T}^{\ast}f\rangle d\mu \\ &\geq A\left\|K^* \mathcal{T}^{\ast} f\right\|^2 \\
& =A\left\|\mathcal{T}^{\ast} K^* f\right\|^2 \\
& \geq A\left\|\left(\mathcal{T}^{+}\right)^{\ast}\right\|^{-2}\left\|K^* f\right\|^2 .
\end{aligned}
$$
On the other hand, we have
$$
\int_{\Omega}\langle  f, \mathcal{T}\mathcal{X}(\omega)\rangle \langle \mathcal{T} \mathcal{Y}(\omega), f\rangle d\mu  =\int_{\Omega}\langle  \mathcal{T}^{\ast}f, \mathcal{X}(\omega)\rangle \langle  \mathcal{Y}(\omega), \mathcal{T}^{\ast}f\rangle d\mu\leq \mu\left\|\mathcal{T}^{\ast} x\right\|^2 \leq \mu\|T\|^2\|x\|^2 .
$$
Hence $$A\left\|\left(\mathcal{T}^{+}\right)^{\ast}\right\|^{-2}\left\|K^* f\right\|^2\leq\int_{\Omega}\langle  f, \mathcal{T}\mathcal{X}(\omega)\rangle \langle \mathcal{T} \mathcal{Y}(\omega), f\rangle d\mu  \leq B\|\mathcal{T}\|^2\|x\|^2 .$$ 
Therefore $(\mathcal{T} \mathcal{X}, \mathcal{T} \mathcal{Y})$ is a continuous $K$-biframe for $\mathcal{R}(\mathcal{T})$.
\end{proof}
\begin{theorem}
Let $K \in \mathcal{B}(\mathcal{H})$ be with a dense range. Let  $(\mathcal{X}, \mathcal{Y})$ be a continuous $K$-biframe and suppose $\mathcal{T} \in \mathcal{B}(\mathcal{H})$ have a closed range. If $(\mathcal{T}\mathcal{X},\mathcal{T} \mathcal{Y})$ and $(\mathcal{T}^{\ast}\mathcal{X},\mathcal{T}^{\ast} \mathcal{Y})$ are continuous $K$-biframes then $\mathcal{T}$ is invertible.
\end{theorem}
\begin{proof}
 Assume that $(\mathcal{T}\mathcal{X},\mathcal{T} \mathcal{Y})$  is a continuous $K$-biframe for $\mathcal{H}$ with frame bounds $A_1$ and $B_1$. Then for every $f\in \mathcal{H}$,
\begin{equation}\label{eq42}
A_1\left\|K^* f\right\|^2 \leq \int_{\Omega}\langle  f, \mathcal{T}\mathcal{X}(\omega)\rangle \langle \mathcal{T} \mathcal{Y}(\omega), f\rangle d\mu  \leq B_1\|x\|^2 .
\end{equation}
Since $K$ has a dense range, $K^*$ is injective. Then from (\ref{eq42}), it follows that $\mathcal{T}^{\ast}$ is injective as $N\left(\mathcal{T}^{\ast}\right) \subseteq N\left(K^*\right)$. Moreover 
$$\mathcal{R}(\mathcal{T})=\mathcal{N}\left(\mathcal{T}^{\ast}\right)^{\perp}=\mathcal{H}.$$ 
Then $\mathcal{T}$ is surjective.

Suppose that $(\mathcal{T}^{\ast}\mathcal{X},\mathcal{T}^{\ast} \mathcal{Y})$  is a continuous $K$-biframe for $\mathcal{H}$ with frame bounds $A_2$ and $B_2$. Then for every $f\in \mathcal{H}$,
\begin{equation}\label{eq43}
A_2\left\|K^* f\right\|^2 \leq \int_{\Omega}\langle  f, \mathcal{T}^{\ast}\mathcal{X}(\omega)\rangle \langle \mathcal{T}^{\ast} \mathcal{Y}(\omega), f\rangle d\mu  \leq B_2\|x\|^2 .
\end{equation}

Since $K$ has a dense range, then $K^*$ is injective. From (\ref{eq43}), $\mathcal{T}$ is injective since $\mathcal{N}(\mathcal{T}) \subseteq \mathcal{N}\left(K^*\right)$.

 we can conclude that $\mathcal{T}$ is bijective. Therefore $\mathcal{T}$ is invertible.
\end{proof}
\begin{theorem}
 Let $K \in \mathcal{B}(\mathcal{H})$  be with a dense range. Let  $(\mathcal{X}, \mathcal{Y})$ be a continuous $K$-biframe for $\mathcal{H}$ and let $\mathcal{T}\in \mathcal{B}(\mathcal{H})$ be co-isometry with $\mathcal{T} K=K \mathcal{T}$. Then $(\mathcal{T}\mathcal{X},\mathcal{T} \mathcal{Y})$ is a continuous $K$-biframe for $\mathcal{H}$.
\end{theorem}
\begin{proof}
Assume that $(\mathcal{X}, \mathcal{Y})$ is a continuous $K$-biframe for $\mathcal{H}$.  Then there are positive constants $0<A \leq B<\infty$ such that
$$
A\left\|K^* f\right\|^2 \leq  \int_{\Omega}\langle  f, \mathcal{X}(\omega)\rangle \langle  \mathcal{Y}(\omega), f\rangle d\mu \leq B\|x\|^2, \text { for all } f \in \mathcal{H} .
$$
Let $\mathcal{T}\in \mathcal{B}(\mathcal{H})$ be co-isometry with $\mathcal{T} K=K \mathcal{T}$. Then for each $f\in \mathcal{H}$, we have
$$
\begin{aligned}
\int_{\Omega}\langle  f, \mathcal{T}\mathcal{X}(\omega)\rangle \langle \mathcal{T} \mathcal{Y}(\omega), f\rangle d\mu  &=\int_{\Omega}\langle  \mathcal{T}^{\ast}f, \mathcal{X}(\omega)\rangle \langle  \mathcal{Y}(\omega), \mathcal{T}^{\ast}f\rangle d\mu\\ &\geq A\left\|K^*\mathcal{T}^{\ast} f\right\|^2 \\
& =A\left\|\mathcal{T}^{\ast} K^* f\right\|^2 \\
& =A\left\|K^* f\right\|^2. 
\end{aligned}
$$
On the other hand, for $f \in \mathcal{H}$ all we have 
$$\int_{\Omega}\langle  f, \mathcal{T}\mathcal{X}(\omega)\rangle \langle \mathcal{T} \mathcal{Y}(\omega), f\rangle d\mu =\int_{\Omega}\langle  \mathcal{T}^{\ast}f, \mathcal{X}(\omega)\rangle \langle  \mathcal{Y}(\omega), \mathcal{T}^{\ast}f\rangle d\mu\leq B\|\mathcal{T}^{\ast}x\|^2 \leq B\|\mathcal{T}\|^2\|x\|^2. $$
Hence $(\mathcal{T}\mathcal{X},\mathcal{T} \mathcal{Y})$ is a continuous $K$-biframe for $\mathcal{H}$.
\end{proof}

\medskip

\section*{Declarations}

\medskip

\noindent \textbf{Availablity of data and materials}\newline
\noindent Not applicable.

\medskip

\noindent \textbf{Competing  interest}\newline
\noindent The authors declare that they have no competing interests.

\medskip

\noindent \textbf{Fundings}\newline
\noindent  Authors declare that there is no funding available for this article.

\medskip

\noindent \textbf{Authors' contributions}\newline
\noindent The authors equally conceived of the study, participated in its
design and coordination, drafted the manuscript, participated in the
sequence alignment, and read and approved the final manuscript.

\end{document}